\documentclass[12pt]{article}
\usepackage[utf8]{inputenc}
\usepackage{amsmath,color,epsfig, amssymb, amsthm}
\usepackage[shortlabels]{enumitem}
\usepackage[left=1in,top=1in,right=1in,bottom=1in]{geometry}
\usepackage{changepage}
\newtheorem{theorem}{Theorem}

\newtheorem{algorithm}[theorem]{Algorithm}

\newtheorem{question}[theorem]{Question}
\newtheorem{corollary}[theorem]{Corollary}
\newtheorem{conjecture}[theorem]{Conjecture}

\begin{document}

\title{The Pyro game: a slow intelligent fire}

\author{M.~E.~Messinger\thanks{Partially supported by NSERC Discovery Grant}, S.~Yarnell\thanks{Supported by a Mount Allison Independent Student Research Award.}}

\date{}

\maketitle

\begin{abstract}  
In the Firefighter problem, a fire breaks out at a vertex of a graph and at each subsequent time step, the firefighter chooses a vertex to protect and then the fire spreads from \emph{each} burned vertex to every unprotected neighbour.  The problem can be thought of as a simplified model for the spread of gossip or disease in a network.  We introduce a new two-player variation called the Pyro game, in which at each step, the fire spreads from \emph{one} burned vertex to all unprotected neighbours of that vertex.  The fire is no longer automated and aims to maximize the number of burned vertices.  We show, that unlike the Firefighter problem, one firefighter can contain a fire on the Cartesian grid in the Pyro game.  We also study both the Pyro Game and the Firefighter Problem on the infinite strong grid and   the complexity of the Pyro game.



\end{abstract}

\section{Introduction to the Pyro game}\label{sec:intro}

The Firefighter problem, introduced in 1995 by Harnell~\cite{hart} is a deterministic model of the spread of fire, gossip, or disease in a network.  Initially, a fire breaks out at a vertex in a graph, and this vertex is said to be \emph{burned}.  At each subsequent step, the firefighter chooses a vertex to \emph{protect} and then the fire acts without intelligence and spreads from each burned vertex to all unprotected neighbouring vertices.  Once a vertex is protected or burned, it remains in that state forever.  In a finite graph, the process terminates when the fire can no longer spread; at this time, any vertex (including protected vertices) that has not been burned is \emph{saved} and naturally, the goal for the firefighter is often to maximize the number of saved vertices.  On an infinite graph, we may either consider the surviving rate, which is the proportion of the vertices that saved when one firefighter protects one vertex at each step; or determine the minimum $k$ such that if $k$ vertices are protected at each time step, the process will terminate (i.e. such that there is some finite step after which no vertex will be burned).  See the survey~\cite{FinMac} or more recent advances in~\cite{Blum,Fomin,Garcia,Gav} for example.

We introduce a variation of the Firefighter problem where during each step, the fire chooses \emph{one} burned vertex and spreads from that vertex to all its unprotected neighbours.  Though the fire is ``slowed'', it is  no longer automated and can act with intelligence.  Formally, the \emph{Pyro game} is a two-player game, played on a connected graph.  At step $0$, the pyro player chooses a vertex and burns it.  At each subsequent time step, the firefighter player (which can be thought of as a set of $k \in \mathbb{N}$ firefighters) protects $k$ unburned vertices and then the pyro player chooses one burned vertex and burns \emph{from} that vertex to its unprotected, unburned neighbours of the vertex.  Once a vertex is burned or protected,  it will remain in that state for the duration of the game.  If the graph is finite, there will a finite step after which no vertex can be burned.  If the graph is infinite, this may or may not be the case.  As a result, we say a fire on an infinite graph can be {\bf contained} if the firefighter player plays optimally, there is some finite step after which no unprotected unburned vertex will be burned.   On a finite graph, the pyro player's goal is to burn as many vertices as possible, while the firefighter player's goal is to maximize the number of vertices saved (i.e. minimize the number of burned vertices).  On an infinite graph, the firefighter player's goal is to contain the fire whereas the pyro player's goal is to avoid being contained.  Henceforth, we will simply refer to the pyro player and the firefighter player as the pyro and the firefighter(s), respectively.

In terms of the maximum number of vertices that can be saved, the Pyro game and the Firefighter problem yield the same result for some graphs and very different results for others.  For example, in both problems the maximum number of vertices that can be saved on a path $P_n = \{v_1,v_2,\dots,v_n\}$ is $n-2$ if the initial burned vertex is a non-leaf and $n-1$ if the initial burned vertex is a leaf.  To do this, at each step, the firefighter simply protects a vertex adjacent to a burned vertex. On the other hand, the maximum number of vertices that can be saved in the Pyro game can also be very different than in the Firefighter problem.  Suppose that vertex $r$ of the graph shown in Figure~\ref{fig:spidey1} (a) is burned at time $t=0$.  Within the dotted lines, a clique on $7$ vertices has been drawn, but we imagine a clique on $m$ vertices (for some large integer $m$).  One can observe that the optimal move for the firefighter, in both the Firefighter problem and the Pyro game, is to protect a neighbour of $r$, say $v_1$, during step $t=1$.  An example of the Firefighter problem is shown in Figure~\ref{fig:spidey1} (b) and an example of the Pyro game is shown in Figure~\ref{fig:spidey1} (c); the white circled vertices indicate protected vertices while the large black circles (and the large black square) indicate burned vertices.

\begin{figure}[htbp]  \[ \includegraphics[width=0.6\textwidth]{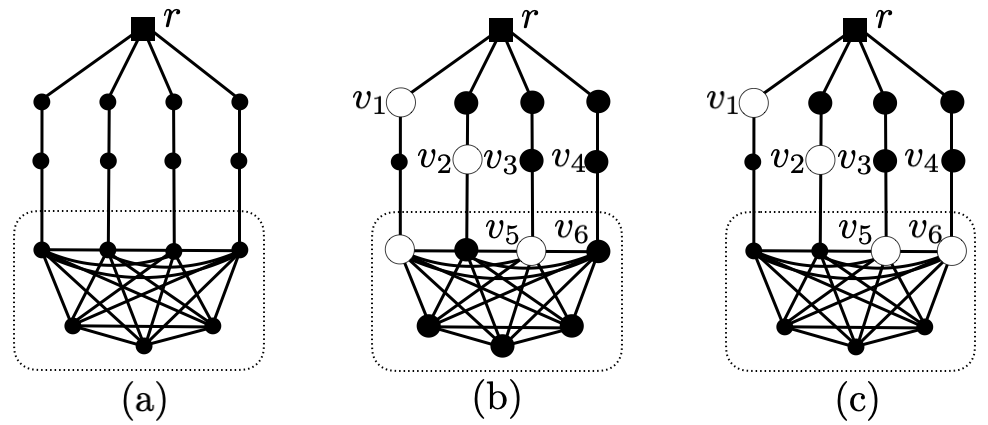} \] 
\vspace{-0.3in}\caption{An example where the maximum number of vertices saved for the Pyro game greatly exceeds that of the Firefighter problem.}
\label{fig:spidey1}
\end{figure}

In the Firefighter problem, the fire spreads to all unprotected neighbours of $r$ at the end of step $t=1$.  Suppose the firefighter protects vertex $v_2$ as shown in Figure~\ref{fig:spidey1} (b).  At the end of step $t=2$ the fire spreads to vertices $v_3$ and $v_4$.  At the beginning of step $t=3$, the firefighter can protect $v_5$, then the fire spreads to $v_6$.  Finally, the firefighter protects $x$, then the fire spreads to all unprotected vertices in $K_m$; this results in only $5$ of $m+9$ vertices being saved (see Figure~\ref{fig:spidey1} (b)).

In the Pyro game, the pyro burns from $r$ to all unprotected neighbours at the end of step $t=1$.  Suppose the firefighter then protects vertex $v_2$.  At the end of step $t=2$, the pyro spreads to either $v_3$ or $v_4$, suppose to $v_3$.  At the beginning of step $t=3$, the firefighter protects $v_5$, then the pyro spreads to $v_4$.  Finally, the firefighter protects $v_6$ and the pyro can no longer spread; this results in $m+3$ of $m+9$ vertices being saved.   (see Figure~\ref{fig:spidey1} (c)).  

The Pyro game is also closely related to Conway's famous Angel Game~\cite{Conway} and Epstein's game of Quadraphage (Square-Eater)~\cite{Gardner}.  The latter is played on a generalized $m \times n$ chessboard with two players.  The Chess player has a single chess piece, the King which starts at the center square of the board (or as close as possible if $mn$ is even).  The Chess player's object is to move the King piece to any square on the edge of the board while the object of the square-eater player is to prevent this from occurring.  The players alternate turns.  On the Chess player's turn, the King is moved to any non-eaten neighbouring square.  On the square-eater player's turn, any $q$ squares are eaten.  The game of Quadraphage with $q=1$ is equivalent to the well-known Angel Game with an angel of power 1.  With optimal play, the square-eater player wins the game of Quadraphage with $q=1$ on a board as small as $33 \times 33$; the proof was shown in~\cite{BCG} and nicely explained in~\cite{kutz}.  The Pyro game differs however, as the pyro can choose any burned vertex from which to burn (not simply the last vertex burned) and when the pyro burns from a vertex $x$, it does not simply spread to one neighbour of $x$, but to all unprotected unburned neighbours of $x$.  


In Section~\ref{sec:Complexity}, we determine the complexity of the Pyro game in terms of the maximum number of vertices that can be saved and show the decision problem is $\mathcal{NP}$-hard for bipartite graphs.  We then consider the Pyro game on the infinite Cartesian and strong grids.  In the Firefighter problem, one firefighter cannot contain a fire on the infinite Cartesian grid, but an average of slightly more than $\frac{3}{2}$ firefighters can.  In Section~\ref{sec:pyroCart}, we show that in the Pyro game, one firefighter can contain a fire on the infinite Cartesian grid.  In Section~\ref{sec:pyroStrong}, we consider the Pyro game on the infinite strong grid; we provide an algorithm for two firefighters to contain the fire.  For completeness, we also consider the Firefighter problem on the strong grid: we prove three firefighters cannot contain a fire on the infinite strong grid and provide an algorithm for four firefighters to contain the fire.

\section{Complexity of the Pyro game}\label{sec:Complexity}

Let $MSV_p(G,r)$ denote the maximum number of vertices of graph $G$ that can be saved if the pyro initially burns vertex $r$ and both players move optimally in the pyro game. By ``optimally'', we mean that the firefighter will aim to maximize the number of saved vertices and the pyro will aim to maximize the number of burned vertices (or equivalently, minimize the number of saved vertices).  Consider the following decision problem:\medskip

\noindent\line(1,0){470}  

{\small{\textbf{PYRO}}}  

\emph{Instance:} A rooted graph $(G,r)$ and an integer $k \geq 1$.  

\emph{Question:} Is $MVS_p(G,r) \geq k$? \vspace{0.05in}

\noindent\line(1,0){470}\bigskip \smallskip

Theorem~\ref{thm:complexity} establishes the problem is $\mathcal{NP}$-hard for bipartite graphs, which implies the result for general graphs.   The transformation is from {\small{\textbf{EXACT COVER BY 3-SETS}}} and the proof is similar to the proof of {\small{\textbf{FIREFIGHTER}}} in~\cite{MW} though the graph construction differs.\medskip

\noindent\line(1,0){470}  

{\small{\textbf{EXACT COVER BY 3-SETS (X3C)}}}  

\emph{Instance:} A set $X$ with $|X|=3q$ and a collection $\mathcal{C}$ of $3$-element subsets of $X$.  

\emph{Question:} Does $\mathcal{C}$ contain an exact cover for $X$? \vspace{0.05in}

\noindent\line(1,0){470}\medskip

\begin{theorem}\label{thm:complexity} {\small{\textbf{PYRO}}} is $\mathcal{NP}$-hard for bipartite graphs.
\end{theorem}


\begin{proof}
Let $\mathcal{C} = \{C_1,C_2,\dots,C_{|\mathcal{C}|}\}$.  We construct a rooted bipartite graph $(G,r)$ and a positive integer $k$ such that at least $k$ vertices of $(G,r)$ can be saved if and only if there is an exact cover of $\mathcal{C}$.  

Begin with a root vertex $r$, and $|\mathcal{C}|$ vertices $c_1,c_2,\dots,c_{|\mathcal{C}|}$ (that represent the elements $C_1,C_2,\dots, C_{|\mathcal{C}|}$ of $\mathcal{C}$) each joined to $r$ by a path of length $q$.  For $i=1,2,\dots,q-1$, we add in edges so that every vertex distance $i$ from $r$ is adjacent to all vertices distance $i+1$ from $r$.  Finally, for each $i,j \in \{1,\dots,|\mathcal{C}|\}$ with $i \neq j$,  if $C_i \cap C_j = \emptyset$ then we join the corresponding vertices $c_i$ and $c_j$ by $9q^4$ paths of length $2$.  Let $P$ be the set of vertices added in this last step of the graph construction.

Without loss of generality, let $\mathcal{C}' = \{C_1,C_2,\dots,C_{|\mathcal{C}'|}\}$ be a subcollection of mutually disjoint $3$-sets of $X$ that is of maximum cardinality.


Suppose that $X$ contains an exact cover by $\mathcal{C}$.   Then every element of $X$ occurs in exactly one member of $\mathcal{C}'$ and $|\mathcal{C}'|=q$.  Without loss of generality, vertices $c_1,c_2,\dots,c_q$ represent the elements of $\mathcal{C}'$.  Since $C_1,C_2,\dots,C_q$ represent mutually disjoint sets, each pair of distinct vertices $c_i$, $c_j$ with $i,j \in \{1,\dots, q\}$ has $9q^4$ paths of length $2$ between them.  In total, there are $\binom{q}{2}(9q^4)$ paths of length $2$.  For $i \in \{1,2,\dots,q\}$, at step $i$, the firefighter protects vertex $c_i$.  There are at least $k=q+\binom{q}{2}(9q^4)$ vertices saved.

Suppose the number of vertices that can be saved is at least $k = q + {q \choose 2}(9q^4)$.  After $q$ steps, suppose $\ell < q$ vertices of $c_1,c_2,\dots,c_{|\mathcal{C}|}$ have been protected.  Then after $q$ steps, at most ${q-1 \choose 2} (9q^4) + (q-\ell)$ vertices in $P$ are protected.  So after $q$ steps, the maximum number of vertices that can be saved is $q+\tbinom{q-1}{2}9q^4 < k$.  As this yields a contradiction, we conclude exactly $q$ vertices of $c_1,c_2,\dots,c_{|\mathcal{C}|}$ are protected after $q$ steps.  Next, suppose that these $q$ vertices do not correspond to a collection of pairwise disjoint $3$-sets.  Then the maximum number of vertices saved is $q+\tbinom{q}{2} 9q^4 - 9q^4 < k$.  Therefore, the $q$ protected vertices must correspond to a collection of pairwise disjoint $3$-sets and $|\mathcal{C}'|=q$.\end{proof} 



\section{The infinite Cartesian grid} \label{sec:pyroCart}

In the Firefighter problem, one firefighter cannot contain a fire on the infinite Cartesian grid~\cite{WangMoeller}.  It was shown in~\cite{3/2} that if an average of $\tfrac{3}{2}+\tfrac{1}{3x+2}$ vertices are protected at each step (protecting at most $2$ vertices at each step), the fire can be contained by time step $12x+7$ for $x \in \mathbb{N}$.  A similar result was independently proven in~\cite{Ng}; they additionally proved that any ratio greater than $\frac{3}{2}$ can be achieved. In both~\cite{3/2} and~\cite{Ng}, the authors also conjectured that if firefighters protect an average of exactly $\tfrac{3}{2}$ vertices at each step, a fire on an infinite Cartesian grid cannot be contained; this was later proven by~\cite{Feldheim}.  By contrast, we next show that in the Pyro game, one firefighter can contain a fire on the Cartesian grid.  

To do this, we provide a simple algorithm to describe the vertices protected at each step $t>0$ that leave the pyro unable to burn a vertex distance $48$ from the vertex $(0,0)$ burned at $t=0$.  The simplicity of the algorithm results from giving the pyro extra ``power'' for the first $44$ steps and by reducing the firefighter's ``power" at every step.  Specifically, for the first $44$ steps, we allow the pyro to burn from {\it each} burned vertex to {\it every} neighbouring vertex.  We additionally impose restrictions on the firefighter and permit the firefighter to only protect vertices that are exactly distance $48$ from the original burned vertex for all $t \geq 1$. The algorithm for protecting vertices, presented below, will prevent the pyro from ever burning a vertex distance $48$ from the original burned vertex.

Let $D_{48}$ be the set of vertices distance $48$ from $(0,0)$ on the infinite Cartesian grid and let $T$ be the set of vertices $(x,y) \in D_{48}$ where $$(-5 \leq x \leq 5) \wedge \Big((43 \leq y \leq 48) \vee (-48 \leq y \leq -43)\Big) \text{~or}$$ $$(-5 \leq y \leq 5) \wedge \Big((43 \leq x \leq 48) \vee (-48 \leq x \leq -43)\Big)$$ as partially illustrated in Figure~\ref{fig:CartT} by the shaded area.

\begin{figure}[htbp]

\[ \includegraphics[width=0.5\textwidth]{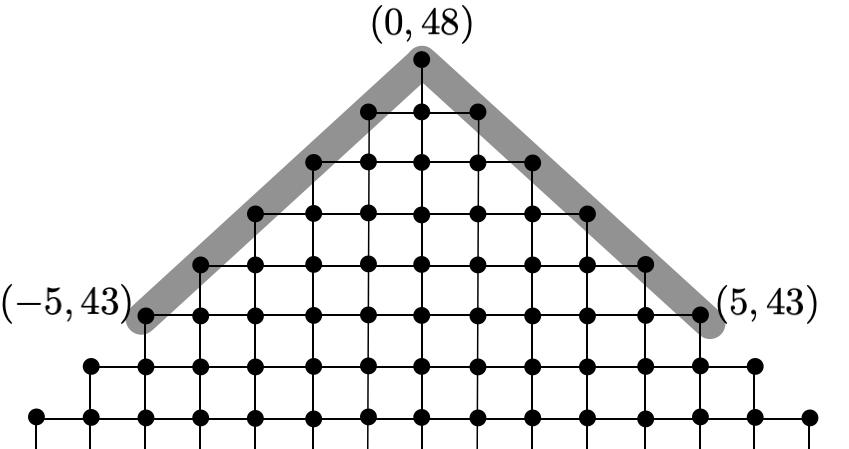} \] 

\label{fig:CartT}

\caption{A subset of vertices of $T$ on the Cartesian grid.}

\end{figure}

If vertex $(x,y)$ is distance $d > 0$ to the nearest vertex in $D_{48}$, then the {\bf threat set} of $(x,y)$, denoted $Th((x,y))$ is defined to be the set of vertices in $D_{48}$ that are distance $d$ from $(x,y)$.  That is, $$Th((x,y)) = \{ (u,v) \in D_{48} ~|~ d((x,y),(u,v))=d\}.$$  We define the {\bf center} of a threat set $Th((x,y))$ to be the vertex (or vertices) in $Th((x,y))$ with the median $x$-coordinate(s).  For example, if $(x',y') \in D_{44}$ with $x' \geq 1$ and $y' \geq 1$ then $$Th((x',y')) = \{(x',y'+4),(x'+1,y'+3),(x'+2,y'+2),(x'+3,y'+1),(x'+4,y')\}$$ with center vertex $(x'+2,y'+2)$.  If $(x',y') \in D_{45}$ with $x' \geq 1$ and $y' \geq 1$ then $$Th((x',y')) = \{(x',y'+3),(x'+1,y'+2),(x'+2,y'+1),(x'+3,y')\}$$ with center vertices $(x'+1,y'+2)$ and $(x'+2,y'+1)$. 

\begin{algorithm}\label{alg:Cart}{\bf The Pyro game with one firefighter on an infinite Cartesian grid.}\medskip

\noindent 1) At step $0$, the pyro chooses a vertex to burn; label this vertex $(0,0)$.\medskip

\noindent 2) For $1 \leq t \leq 44$, at step $t$, the firefighter protects an unprotected vertex of $T$ and then the pyro burns from every vertex distance $t-1$ from the origin, to every vertex distance $t$ from the origin.\medskip

\noindent 3) At the beginning of step $45$, the firefighter protects the last unprotected vertex of $T$.\medskip

\noindent 4) For $t \geq 45$, during step $t$, the pyro will choose one burning vertex $(x,y)$ and spread from $(x,y)$ to all unprotected unburned neighbours:

\begin{enumerate}[a)] \item if the pyro burns from a vertex in $D_{44}\cup D_{45}\cup D_{46}$ during step $t$, then during step $t+1$, the firefighter will protect an unprotected vertex of $Th((x,y))$, closest to the center, arbitrarily breaking ties.  If all vertices in $Th((x,y))$ were protected prior to step $t+1$, the firefighter protects an arbitrary vertex in $D_{48}$.\medskip

\item if the pyro burns from a vertex $(x,y) \in D_{47}$ during step $t$, then if~ $\exists (u,v) \in N((x,y))\backslash D_{48}$ that was unburned prior to step $t$, protect a vertex in $Th((u,v))$ closest to the center, arbitrarily breaking ties.  If no such $(u,v)$ exists or every vertex in $Th((u,v))$ was protected prior to step $t+1$, protect an arbitrary vertex in $D_{48}$.\end{enumerate}\end{algorithm}

\begin{theorem}Using Algorithm~\ref{alg:Cart}, one firefighter can prevent the pyro from burning any vertex distance $48$ from the original burned vertex $(0,0)$ on the infinite Cartesian grid.\end{theorem}

\begin{proof}  For a contradiction, suppose the pyro burns from a vertex in $D_{47}$ during step $t+1$ to spread to a vertex of $D_{48}$.  We further assume that step $t+1$ is the least step by which a vertex of $D_{48}$ can be burned.  Then the pyro must have burned from a vertex $(x,y) \in D_{46}$ during step $t$.  By 4) a), the firefighter protects one vertex from the set $Th((x,y))$ during step $t+1$ (or an arbitrary vertex of $D_{48}$ if the vertices in $Th((x,y))$ were all previously protected).  Using Algorithm~\ref{alg:Cart}, we next show that by the end of step $t+1$, all vertices in the set $Th((x,y))$ will be protected which proves that if the pyro burns from a vertex in $N((x,y)) \cap D_{47}$ during step $t+1$, no vertex of $D_{48}$ will be burned during step $t+1$.\medskip

By 2) and 3), after the firefighter has protected a vertex during step $45$, every vertex of $T$ will have been protected.  Observe that for $(x,y) \in D_{46}$ where $-3 \leq x \leq 3$ or $-3 \leq y \leq 3$, by 2) and 3), every vertex in the set $Th((x,y))$ protected by the end of step $45 < t+1$.  If $(x,y) \in D_{46}$ where $x = \pm 4$ or $y=\pm 4$, then by 2) and 3), two of the three vertices of $Th((x,y))$ are protected by the end of step $45< t+1$ and by 4) a), all vertices in $Th((x,y))$ are protected by the end of step $t+1$.  Thus, we may assume $x \notin \{-4,-3,-2,-1,0,1,2,3,4\}$ and $y \notin \{-4,-3,-2,-1,0,1,2,3,4\}$ for the remainder of the proof.  This implies $|Th((x,y))| = 3$.

Without loss of generality, suppose $x \geq 5$ and $y \geq 5$ and consider the vertex in $N((x,y)) = \{(x-1,y),(x,y+1),(x+1,y),(x,y-1)\}$ that the pyro burned from most recently prior to step $t$.  If the pyro burned from $(x,y+1)$ (or $(x+1,y)$) during some step $t^* < t$ then by the end of step $t^*+1$, vertices $(x,y+2)$ and $(x+1,y+1)$ (or $(x+1,y+1)$ and $(x+2,y)$) were protected; otherwise the pyro could have burned a vertex of $D_{48}$ during step $t^*+1 < t+1$ which contradicts our assumption that the earliest step by the pyro can burn a vertex of $D_{48}$ is step $t+1$.  

Suppose the pyro burned from $(x-1,y)$ during some step $t^*<t$.  Note $Th((x-1,y)) = \{(x,y+3),(x+1,y+2),(x+2,y+1),(x+3,y)\}$.  If both center vertices $(x+1,y+2),(x+2,y+1)$ were protected prior to step $t^*+1$, then the firefighter protects $(x,y+3)$ or $(x+3,y)$ during step $t^*+1$.  If the firefighter protects $(x,y+3)$ during step $t^*+1$, then the firefighter protects $(x+3,y)$ during step $t+1$ and so that all vertices of $Th((x,y))$ are protected before the pyro burns during step $t+1$. 

Thus, we assume at most one of $(x+1,y+2),(x+2,y+1)$ was protected prior to step $t^*+1$.  If the pyro burned from $(x-1,y+1)$ during step $t^{**}< t^*$, then the vertices in $Th((x-1,y+1))$ are all protected by the end of step $t^{**}+1$, otherwise the pyro could have burned a vertex of $D_{48}$ prior to step $t+1$.  By 4) a), the vertices of $Th((x,y)) = Th((x-1,y+1)) \cup \{(x,y+2)\}$ are protected by the end of step $t+1$.

If the pyro burned from $(x-2,y)$ during step $t^{**}<t^*$, the by 4) a), as $(x,y+2)$ is the center of $Th((x-2,y))$, it is protected by the end of step $t^{**}+1$.  As a result, by 4) a), $(x+1,y+1) \in Th((x-1,y))$ is protected by the end of step $t^*+1$ and consequently, $(x+2,y) \in Th((x,y))$ is protected by the end of step $t+1$.  Observe then, that the vertices of $Th((x,y)) = \{(x,y+2),(x+1,y+1),(x+2,y)\}$ are all therefore protected by the end of step $t+1$.   A very similar argument shows that if the pyro burned from $(x-1,y-1)$ during step $t^{**}<t^*$ then the vertices of $Th((x,y))$ are all protected by the end of step $t+1$.

The argument to show that if the pyro burned from $(x,y-1)$ instead of $x-1,y)$ during some step $t^*<t$ then all vertices of $Th((x,y))$ are protected by the end of step $t+1$ is extremely similar to the previous case and has therefore been omitted.\end{proof}

Though we proved that one firefighter can prevent the pyro from burning a vertex distance $48$ from the original burned vertex, this is far from optimal and we conjecture a much stronger result.\footnote{The authors claim to have a proof of the conjecture, but as it is an extremely long and tedious proof, they do not believe it is worth publishing and hope a more clever proof can be found.}

\begin{conjecture} One firefighter can prevent the pyro from burning any vertex distance $7$ from the original burned vertex $(0,0)$ on the infinite Cartesian grid.\end{conjecture}

In the (original) Firefighter problem, two firefighters are necessary to contain a fire on the 2-dimensional Cartesian grid, but are not sufficient to contain a fire on the 3-dimensional Cartesian grid: for the 3-dimensional Cartesian grid $5$ firefighters are necessary~\cite{hartke}.   
\begin{question}\label{thm:3d}
What is the minimum integer $f$ such that if firefighters protect $f$ vertices at each step, the firefighters can contain the pyro on a 3-dimensional Cartesian grid.
\end{question}

\section{The infinite strong grid}\label{sec:pyroStrong}

\subsection{The Pyro game on an infinite strong grid}\label{pyro:strong}

In this section, we show that in the Pyro game, two firefighters (protecting two vertices at each step $t>0$) suffice to contain a fire on the strong grid.  To do this, we provide a simple algorithm to describe the vertices protected at each step $t>0$ that leave the pyro fire unable to burn a vertex distance $29$ from the vertex $(0,0)$ burned at $t=0$.  Like Algorithm~\ref{alg:Cart}, the simplicity of the algorithm results from giving the pyro extra ``power'' for the first $25$ steps and by reducing the firefighters ``power" at every step.  In particular, for the first $25$ steps, we allow the pyro to burn from {\it each} burned vertex to {\it every} neighbouring vertex.  We additionally impose restrictions on the firefighters and allow them to only protect vertices that are exactly distance $29$ from the original burned vertex. 

Let $T$ be the set of vertices $(x,y)$ where \begin{itemize} \item $x= \pm 29$ and $(23 \leq y \leq 29) \vee (-29 \leq y \leq -23)$; or 

\item $y = \pm 29$ and $(23 \leq x < 29) \vee (-29 < x \leq 23)$. \end{itemize} 

Set $T$ is partially illustrated in Figure~\ref{fig:strongT} by the shaded area.  The algorithm below dictates that that during each of the first $26$ steps, the firefighters protect $2$ of the $52$ vertices of $T$ and for $t > 26$, describes how the firefighters protect vertices of set $D = \{(x,y): dist((x,y),(0,0))=29\} \backslash T$.

\begin{figure}[htbp]
\[ \includegraphics[width=0.375\textwidth]{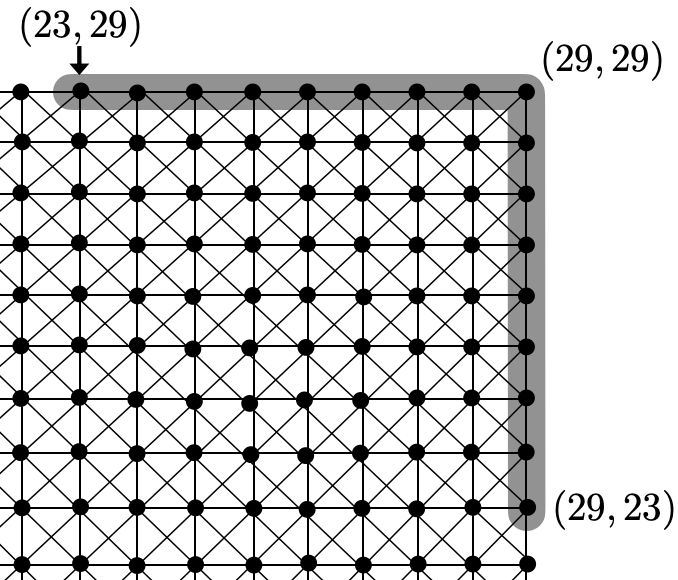} \]
\caption{A subset of vertices of $T$ on the strong grid.} 

\label{fig:strongT}
\end{figure}

\begin{algorithm}\label{alg:strong}{\bf The Pyro game with two firefighters on an infinite strong grid.}\medskip

\noindent 1) At step $0$, the pyro chooses a vertex to burn; label this vertex $(0,0)$.\medskip

\noindent 2) For $1 \leq t \leq 25$, at step $t$, the firefighters protect any two unprotected vertices of $T$ and then the fire burns from every vertex distance $t-1$ from the origin, to every vertex distance $t$ from the origin.\medskip

\noindent 3) At the beginning of step $26$, the firefighters protect the last two unprotected vertices of $T$.\medskip

\noindent 4) For $t \geq 26$, during step $t$, the pyro will choose one burning vertex $(x,y)$ and spread from that vertex to all unprotected unburned neighbours.  Observe that $(x+i,y) \oplus (x,y+i) \in D$ for some $i \in \{-4,-3,-2,-1,1,2,3,4\}$ since $|x| \geq 25$ and $|y| \geq 25$. During during step $t+1$, the firefighters will proceed as follows:\medskip

\hspace{0.45in}a) If $(x,y+i) \in D$, the firefighters protect two unprotected vertices of $$A_t = \{ (x-2,y+i),(x-1,y+i),(x,y+i),(x+1,y+i),(x+2,y+i)\}$$ that are closest to $(x,y+i)$, arbitrarily breaking ties.  (If there is only one unprotected vertex in $A_t$, the firefighters protect that vertex and an arbitrary vertex of $D$; and if there are no unprotected vertices in $A_t$, the firefighters protect two arbitrary vertices of $D$.)\medskip

\hspace{0.45in}b) If $(x+i,y) \in D$, the firefighters protect two unprotected vertices of $$B_t = \{ (x+i,y-2),(x+i,y-1),(x+i,y),(x+i,y+1),(x+i,y+2)\}$$ that are closest to $(x+i,y)$, arbitrarily breaking ties.  (If there is only one unprotected vertex in $B_t$, the firefighters protect that vertex and an arbitrary vertex of $D$; and if there are no unprotected vertices in $B_t$, the firefighters protect arbitrary vertices of $D$.)\end{algorithm}

We next prove the algorithm will prevent any vertex distance $29$ from being burned.

\begin{theorem}By following Algorithm~\ref{alg:strong}, in the Pyro game two firefighters suffice to contain a fire on the strong grid.\end{theorem}

\begin{proof}We will use Algorithm~\ref{alg:strong} to prove a stronger result, namely that if the pyro is given additional ``power'' to burn from each burning vertex to each neighbouring vertex during the first $25$ steps, two firefighters can prevent the pyro from burning any vertex distance $29$ from the original burned vertex.\medskip

Consider any step $t \geq 26$: during step $t$, the pyro will choose one burning vertex $(x,y)$ and burn from $(x,y)$ to all neighbouring vertices.  Since we assumed that every vertex within distance $26$ of the original burned vertex $(0,0)$ is burned by the end of step $26$, we know that either $(x,y+i) \in D$ or $(x+i,y) \in D$ for some $i \in \{-4,-3,-2,-1,1,2,3,4\}$.  And from the definitions of $T$ and $D$, exactly one of $(x,y+i)$, $(x+i,y)$ is in $D$.\medskip

Consider 4) a) of Algorithm~\ref{alg:strong}, where $(x,y+i) \in D_t$ and wlog, suppose $i \in \{1,2,3,4\}$.\medskip  

If $i=4$, then after the firefighters have moved during step $t+1$, clearly at least two vertices of $A_t$ have been protected.\medskip
 
If $i=3$, then consider the neighbour(s) of $(x,y)$ that were burned at some step(s) prior to $t$.  The Pyro burned from least one of $(x-1,y),(x-1,y+1),(x,y+1),(x+1,y+1),(x+1,y),(x+1,y-1),(x,y-1),(x-1,y-1)$ at an earlier step.  But this would mean at least one of $(x-1,y+3),(x,y+3),(x+1,y+3)$ was protected at an earlier step.  Thus, after the firefighters have protected vertices during step $t+1$, vertices $(x-1,y+3),(x,y+3),(x+1,y+3)$ have all been protected.\medskip

If $i=2$, then consider the neighbour(s) of $(x,y)$ that were burned at some step(s) prior to $t$.  The Pyro burned from least one of $(x-1,y),(x-1,y+1),(x,y+1),(x+1,y+1),(x+1,y),(x+1,y-1),(x,y-1),(x-1,y-1)$ at an earlier step; call this vertex $(x',y')$ and suppose it was burned at step $t'$.   But this would mean at least two of $A_t$ were protected at an earlier step than $t+1$.  And, the Pyro burned some vertex in the neighbourhood of $(x',y')$ at some step $t''<t'$.  If $x'' \in \{x-1,x,x+1\}$ then at least two vertices of $A_t$ were protected at step $t''+1$.  If $x'' \in \{x-2,x+2\}$, then at least one of $(x-2,y+i)$, $(x-2,y+i)$ was protected at step $t''+1$.  In any event, after the firefighters protect vertices during step $t+1$, every vertex of $A_t$ has been protected.\medskip

If $i=1$, then consider the neighbour(s) of $(x,y)$ that were burned at some step(s) prior to $t$.  The Pyro burned from least one of $(x-1,y),(x+1,y),(x+1,y-1),(x,y-1),(x-1,y-1)$ at an earlier step; call this vertex $(x',y')$ and suppose it was burned at step $t'$.  But this would mean at least two of $A_t$ were protected by step $t'+1$.  And, the Pyro burned some vertex (call it $(x'',y'')$) in the neighbourhood of $(x',y')$ at some step $t''<t'$.  If $x'' \in \{x-1,x,x+1\}$ then at least two vertices of $A_t$ were protected at step $t''+1$.  If $x'' \in \{x-2,x+2\}$ then at least one of $(x-2,y+i),(x+2,y+i)$ was protected at step $t''+1$.  In any event, after the firefighters protect vertices during step $t+1$, every vertex of $A_t$ has been protected.\medskip

Thus, we can see that if the pyro burns from some vertex $(x,y)$ at step $t$ where $i \in \{-2,-1,1,2\}$, then the pyro cannot burn a vertex of $D$ in step $t+1$.  Since 4) b) of Algorithm~\ref{alg:strong} is extremely similar to a), we omit the remainder of the proof to avoid a redundant explanation.\end{proof} 

Though two firefighters suffice to contain a fire on the strong grid, it remains unknown as to whether one firefighter suffices. We conjecture it is not possible. 

\begin{conjecture}In the Pyro game, one firefighter cannot contain a fire on the strong grid.\end{conjecture}

\subsection{The Firefighter problem on an infinite strong grid}

In Section~\ref{pyro:strong}, we considered the Pyro game on the infinite strong grid.  In order to better understand the differences between the Pyro game and the Firefighter problem, we next consider the Firefighter problem on the infinite square grid.  We prove that four, but not three firefighters can contain a fire on the infinite strong grid.

\begin{theorem} In the Firefighter problem, firefighters protecting three vertices at each step cannot contain a fire on the infinite strong grid.\end{theorem}

\begin{proof} For a contradiction, suppose firefighters protecting three vertices at each step suffice to contain a fire on the infinite strong grid.  Note that once the fire can no longer spread, any protected vertex $(x,y)$ that is not adjacent to a burned vertex was unnecessarily protected: if $(x,y)$ had not been protected, the fire would still have been contained.  Additionally, note that any protected vertex $(x,y)$ that is adjacent to a burned vertex but not an unprotected saved vertex was also unnecessarily protected: if $(x,y)$ had not been protected, the fire would still have been contained.  Thus, we may further, assume the firefighters protect a {\it minimal} set of vertices to contain the fire.  Consequently, once the fire has been contained, every protected vertex is adjacent to at least one unprotected saved vertex and at least one burned vertex.  Let $S$ be a minimal set of vertices that can be protected in order to contain the fire.\bigskip

Suppose that once the fire has been contained, $(x,y), (x,y+1),(x+1,y)$ are all in $S$ and $(x+1,y+1)$ was saved but not in $S$.  Then during the step where $(x,y)$ was protected, the firefighters could alternately have protected $(x+1,y+1)$ and still contained the fire.  Thus, $S \backslash \{(x,y)\} \cup \{(x+1,y+1)\}$ is a minimal set of vertices that can be protected in order to contain the fire.  We refer to the action of updating $S$ to $S \backslash \{(x,y)\} \cup \{(x+1,y+1)\}$ as ``popping a corner", as illustrated in Figure~\ref{fig:pop}.  

\begin{figure}[htbp]
\[ \includegraphics[width=0.3\textwidth]{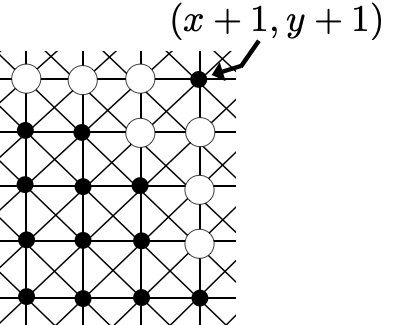}\]
\caption{An example of a set of protected (white) vertices before a corner is popped.} 

\label{fig:pop} 
\end{figure}

Similarly, updating $S$ in each of the three ways below will also be referred to as ``popping a corner'': \begin{itemize} 

\item if once the fire has been contained, $(x,y),(x,y+1),(x-1,y)$ are all in $S$ and $(x-1,y+1)$ was saved but not in $S$ then $S \backslash \{(x,y)\} \cup \{(x-1,y+1)\}$ is a minimal set of vertices that can be protected in order to contain the fire.  

\item if once the fire has been contained, $(x,y),(x,y-1),(x-1,y)$ are all in $S$ and $(x-1,y-1)$ was saved but not in $S$ then $S \backslash \{(x,y)\} \cup \{(x-1,y-1)\}$ is a minimal set of vertices that can be protected in order to contain the fire.  

\item if once the fire has been contained, $(x,y),(x,y-1),(x+1,y)$ are all in $S$ and $(x-1,y-1)$ was saved but not in $S$ then $S \backslash \{(x,y)\} \cup \{(x-1,y-1)\}$ is a minimal set of vertices that can be protected in order to contain the fire.  
\end{itemize}

If we iteratively apply the means of updating the set of protected vertices, effectively ``popping all the corners", then obviously the set of protected vertices will form a rectangle.  More precisely, this implies there exists a strategy for the firefighters contain the fire so that the following sets of vertices are protected: $$\{(u,\ell_1): k_2 \leq u \leq k_1\},~~\{(u,-\ell_2):-k_2 \leq u \leq k_1\},$$ $$\{(-k_2,v): -\ell_2 < v < \ell_1\},~~\{(k_1,v):-\ell_2 < v < \ell_1\}$$ for some non-negative integers $k_1,\ell_1,k_2,\ell_2$.\\

We next obtain a contradiction and show the firefighters cannot actually contain the fire.  Without loss of generality, suppose $k_2 \geq k_1$, $\ell_2 \geq \ell_1$, and $\ell_1 \geq k_1$. 

First,  suppose $\ell_1 \leq k_2$ and consider the vertices burned and protected by the end of step $t = \ell_1$.  To prevent the fire from burning a vertex with $x$-coordinate $k_1+1$ or a vertex with $y$-coordinate $\ell_1+1$, the firefighters must protect at least $2\ell_1+1$ vertices with $x$-coordinate $\ell_1$ and $k_1+\min\{\ell_1,k_2\}=k_1+\ell_1$ vertices with $y$-coordinate $\ell_1$ by the end of step $t=\ell_1$. However, this yields a total of at least $3\ell_1+k+1+1$ vertices and provides a contradiction as the firefighters can only protect $3\ell_1$ vertices by the end of step $\ell_1$.  

Second, suppose $k_2 < \ell_1$.  Then $k_1 \leq k_2 < \ell_1 \leq \ell_2$.  To prevent the fire from burning a vertex with $x$-coordinate $-k_2-1$ or $k_1+1$, the firefighters must protect at least $2k_2+1$ vertices with $x$-coordinate $-k_2$ and $2k_2+1$ vertices with $x$-coordinate $k_1$.  This yields a total of at least $4k_2+2$ protected vertices and provides a contradiction as the firefighters can protect only $3k_2$ vertices by the end of step $k_2$.\end{proof} 

It is easy to see that four firefighters suffice to contain a fire on the infinite strong grid and this is illustrated in Figure~\ref{fig:4ff}.  The square vertex indicates the original burned vertex; the white vertices with numbers indicate protected vertices and the step in which each vertex is protected.  Thus, we can see that four firefighters can contain a fire by the end of step $t=8$.

\begin{corollary} Four firefighters suffice to contain a fire on the infinite strong grid.\end{corollary}

\begin{figure}[htbp]
\[ \includegraphics[width=0.325\textwidth]{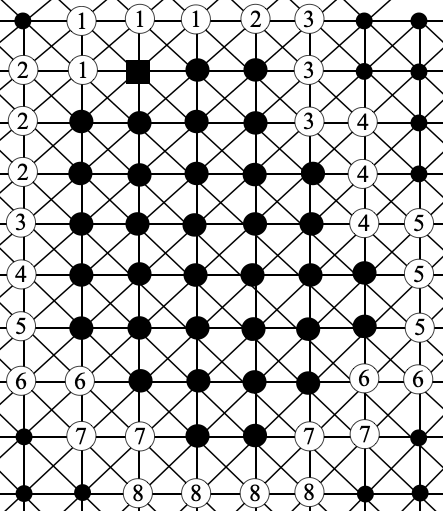}\]
\caption{Containing a fire by protecting four vertices at each step.} 

\label{fig:4ff} 
\end{figure}


\end{document}